\documentclass[11pt]{article}

\usepackage{verbatim,latexsym,amsfonts,amsmath,amssymb,graphicx,fancyhdr,hyperref,asymptote,enumitem }
\usepackage{appendix,latexsym,amsfonts,amsmath,amssymb,graphicx,hyperref,amsthm,soul,verbatim,authblk,enumitem}
\usepackage[framemethod=tikz]{mdframed}

\newcommand{\EE}{\mathbb{E}}
\newcommand{\PP}{\mathbb{P}}

\newcommand{\R}{\mathbb{R}}
\newcommand{\C}{\mathbb{C}}

\newcommand{\HH}{\mathbb{H}}
\newcommand{\N}{\mathbb{N}}
\newcommand{\D}{\mathbb{D}}

\newcommand{\TT}{\mathbb{T}}

\newcommand{\pa}{\partial}

\newcommand{\F}{{\cal F}}

\def\eps{\varepsilon}
\def\til{\widetilde}
\def\ha{\widehat}
\def\sem{\setminus}
\def\lin{\overline}

\def\M{{\cal M}}

\DeclareMathOperator{\ccap}{cap} \DeclareMathOperator{\Cont}{Cont}

 \DeclareMathOperator{\diam}{diam}
\DeclareMathOperator{\dist}{dist} 
 
\DeclareMathOperator{\Imm}{Im } \DeclareMathOperator{\Ree}{Re }

\DeclareMathOperator{\mA}{m}

\theoremstyle{plain}
\newtheorem{Theorem}{Theorem}[section]
\newtheorem{Lemma}[Theorem]{Lemma}
\newtheorem{Corollary}[Theorem]{Corollary}
\newtheorem{Proposition}[Theorem]{Proposition}
\theoremstyle{definition}
\newtheorem{Definition}[Theorem]{Definition}
\newtheorem{Remark}[Theorem]{Remark}
\numberwithin{equation}{section}
\newcommand{\BGE}{\begin{equation}}
\newcommand{\BGEN}{\begin{equation*}}
\newcommand{\EDE}{\end{equation}}
\newcommand{\EDEN}{\end{equation*}}

\begin{document}
\title{Optimal H\"older Continuity and Dimension Properties\\ for SLE with Minkowski Content Parametrization}
\author{Dapeng Zhan\thanks{Research partially supported by NSF grant  DMS-1056840 and Simons Foundation grant \#396973.}}
\affil{Michigan State University}
\maketitle

\begin{abstract}
	We make use of the fact that a two-sided whole-plane Schramm-Loewner evolution (SLE$_\kappa$) curve $\gamma$ for $\kappa\in(0,8)$ from $\infty$ to $\infty$ through $0$ may be parametrized by its $d$-dimensional Minkowski content, where $d=1+\frac\kappa 8$, and become a self-similar process of index $\frac 1d$ with stationary increments.
	We prove that such $\gamma$ is locally $\alpha$-H\"older continuous for any $\alpha<\frac 1d$. In the case $\kappa\in(0,4]$, we show that $\gamma$ is not locally $\frac 1d$-H\"older continuous. We also prove that,   for any deterministic closed set $A\subset \R$, the Hausdorff dimension of $\gamma(A)$ almost surely equals $d$ times the Hausdorff dimension of $A$.
\end{abstract}

\section{Introduction}
\subsection{Overview}
The Schramm-Loewner evolution (SLE$_\kappa$), introduced by Oded Schramm in 1999 (\cite{Sch}), is a one-parameter ($\kappa\in(0,\infty)$) family of probability measures on non-self-crossing curves, which has received a lot of attention since then.
It has been shown that, modulo time parametrization, the interface of several discrete lattice models have SLE$_\kappa$ with different parameters $\kappa$ as their scaling limits. The reader may refer to \cite{Law1,RS} for basic properties of SLE.

The regularity property of SLE curves have been studied by a number of authors. Rohde and Schramm proved in \cite{RS} that, for $\kappa\ne 8$, an SLE$_\kappa$ curve exists and is H\"older continuous in its original capacity parametrization. Lind improved the estimates by Rohde and Schramm and derived a better H\"older exponent for SLE in the capacity parametrization (\cite{Lind}), which was later proved (\cite{Optimal}) to be optimal.

People tried to find better H\"older exponents of SLE curves with other parametrizations. It is know that (\cite{RS,Bf}) the Hausdorff dimension of SLE$_\kappa$ is $\min\{1+\frac\kappa 8,2\}=:d$. Thus an SLE$_\kappa$ curve  can not be reparametrized to be H\"older continuous of any order greater than $1/d$. On the other hand, using a variation estimate, Werness proved (\cite{Werness}) that, for $\kappa\le 4$, for any $\alpha<1/d$, an SLE$_\kappa$ curve may be reparametrized to be H\"older continuous of order $\alpha$. The result was later extended to all $\kappa\ne 8$ by Friz and Tran (\cite{FT}).

The natural parametrization of SLE$_\kappa$ is defined for $\kappa\in(0,8)$ (\cite{LS,LZ}), and is expected to be the scaling limit of the natural length of various lattice models, while the convergence has been established for loop-erased random walk (\cite{LERW-NP2,LERW-NP4}). Lawler and Rezaei proved (\cite{LR}) that the natural parametrization of an SLE$_\kappa$ curve   agrees with the $d$-dimensional Minkowski content of the curve. So the natural parametrization is determined by the SLE$_\kappa$ curve.


It was proven in \cite{LR} that the transition function $\Theta$ between the capacity parametrization and the natural parametrization is H\"older continuous, where $\Theta_t$ is defined to be the Minkowski content of the SLE$_\kappa$ curve under capacity parametrization during the time interval $[0,t]$, i.e., if $\gamma$ is an SLE curve with capacity parametrization, then  $\gamma\circ \Theta^{-1}$ is the same SLE curve with natural parametrization. To get any H\"older regularity  of SLE in the natural parametrization based on the H\"older regularity  of SLE in the capacity parametrization, one needs the H\"older continuity of $\Theta^{-1}$, which was not derived in \cite{LR}.

\subsection{Main results}
We fix $\kappa\in(0,8)$ and let $d=1+\frac{\kappa}{8}$.
The main purpose of this paper is to solve a conjecture proposed by Greg Lawler: an SLE$_\kappa$ curve under Minkowski content parametrization is locally H\"older continuous of any order less than $1/d$. We will discard the boundary effect, and focus on the H\"older continuity of any compact subcurve that is bounded away from the boundary.

There are several different versions of SLE. Thanks to the local equivalence between them (cf.\ \cite{LSW1}), we have the freedom to work on any version of SLE$_\kappa$ in any domain.
The main object studied in this paper is a two-sided whole-plane SLE$_\kappa$ curve $\gamma$ in the Riemann sphere $\ha\C$ from $\infty$ to $\infty$ passing through $0$. The nice property of such SLE curve is that the Minkowski content parametrization makes it a self-similar process with stationary increment. 

\begin{Proposition} \cite[Corollary 4.7]{loop}
	A two-sided whole-plane SLE$_\kappa$ curve $\gamma$ from $\infty$ to $\infty$ passing through $0$ may be parametrized by its $d$-dimensional Minkowski content, and become a self-similar process of index $\frac 1d$  with stationary increments. \label{stationary*}
\end{Proposition}

We will refer the $\gamma$ in Proposition \ref{stationary*} as an sssi SLE$_\kappa$ curve. Here we say that a curve $\gamma$ is parametrized by $d$-dimensional Minkowski content, if for any $a<b$ in the definition domain of $\gamma$, the $d$-dimensional Minkowski content of $\gamma([a,b])$ is $b-a$. A random curve $\gamma$ is called a self-similar process of index $H$  with stationary increments if it is defined on $\R$ with $\gamma(0)=0$ such that (i) for any $a>0$, $(\gamma(at))$ has the same law as $(a^H\gamma(t))$; and (ii) for any $a\in\R$, $(\gamma(a+t)-\gamma(a))$ has the same law as $(\gamma(t))$.

\begin{Theorem}
	An sssi SLE$_\kappa$ curve is a.s.\ locally H\"older continuous of any order less than $1/d$. \label{Theorem-Holder*}
\end{Theorem}

The following theorem resembles Mckean's dimension theorem for Brownian motion (\cite{McK55}). We use $\dim_H$ to denote the Hausdorff dimension. It is closely related to the work in \cite{GHM2,GHM1}, which proves analogous results when the Liouville quantum gravity variant of the Minkowski content is used.

\begin{Theorem}
	For an sssi SLE$_\kappa$ curve $\gamma$ and any deterministic closed set $A\subset \R$, almost surely \BGE \dim_H(\gamma(A))=d\cdot \dim_H(A).\label{Mckean}\EDE
\label{Theorem-Mckean*}
\end{Theorem}

\begin{Theorem}
	If $\kappa\in(0,4]$, an sssi SLE$_\kappa$ curve $\gamma$ is a.s.\ not $1/d$-H\"older continuous   on any open interval. \label{not-Holder}
\end{Theorem}

\begin{Corollary}
  Let $\gamma$ be a chordal SLE$_\kappa$, radial SLE$_\kappa$, or SLE$_\kappa(\rho)$ curve, $\kappa\in(0,8)$, in any simply connected domain $D$. Let $\beta$ be a compact subcurve of $\gamma$, which has positive distance from $\pa D$ and the set of marked  points of $\gamma$ (including the force points and the target). If $\beta$ is parametrized by its $d$-dimensional Minkowski content measure, then it is a.s.\  H\"older continuous of any order less than $1/d$, and a.s.\ not $1/d$-H\"older continuous when $\kappa\le 4$.
\end{Corollary}
\begin{proof}
   This follows from Theorems \ref{Theorem-Holder*} and \ref{not-Holder}, the conformal covariance of Minkowski content (Proposition \ref{conformal-content}) and the local equivalence between different versions of SLE.
\end{proof}

\begin{Remark}
Theorems \ref{Theorem-Holder*}, \ref{Theorem-Mckean*} and \ref{not-Holder} also hold for $\kappa=8$. See Remark \ref{Remark-last}.
We expect that Theorem \ref{not-Holder} also holds for $\kappa\in(4,8)$.
It is not known to the author whether an SLE$_\kappa$ curve possesses some parametrization that makes it exactly locally $1/d$-H\"older continuous.
\end{Remark}

\subsection{Some proofs}	
The proofs of the main theorems follow some standard arguments. We now give the proofs of Theorems \ref{Theorem-Holder*} and \ref{Theorem-Mckean*} assuming that the following finite moment lemma holds.

\begin{Lemma}
	If $\gamma$ is an sssi SLE$_\kappa$ curve, then
	for any $c\in(-d,\infty)$, $\EE[|\gamma(1)|^c]<\infty$. \label{integrable*}
\end{Lemma}

\begin{proof}[Proof of Theorem \ref{Theorem-Holder*}]
	Let $\alpha<1/d$. Let $I\subset\R$ be a bounded interval.
	We use the  Garsia-Rodemich-Rumsey inequality (\cite{inequality}) in the following form: For any $p>1/\alpha$, there exists a constant $C_{\alpha,p}\in(0,\infty)$ such that for any $s,t\in I$,
	\BGE |\gamma(t)-\gamma(s)|^p\le C_{\alpha,p}|t-s|^{\alpha p-1}\int_I\!\int_I \frac{|\gamma(s)-\gamma(t)|^p}{|s-t|^{\alpha p+1}}\,dsdt.\label{inequality*}\EDE
	Using Proposition \ref{stationary*}, Lemma \ref{integrable*} and Fubini Theorem, we see that
	$$\EE\Big[\int_I\!\int_I \frac{|\gamma(s)-\gamma(t)|^p}{|s-t|^{\alpha p+1}}\,dsdt\Big]
	=\int_I\!\int_I \frac{\EE[|\gamma(s)-\gamma(t)|^p]}{|s-t|^{\alpha p+1}}\,dsdt $$
	$$
	=\int_I\!\int_I \frac{|s-t|^{p/d}\EE[|\gamma(1)|^p]}{|s-t|^{\alpha p+1}}\,dsdt=\EE[|\gamma(1)|^p]\int_I\!\int_I  |s-t|^{(1/d-\alpha)p-1}dsdt<\infty.$$
	Thus, a.s.\ $\int_I\int_I \frac{|\gamma(s)-\gamma(t)|^p}{|s-t|^{\alpha p+1}}dsdt<\infty$. From (\ref{inequality*}) we see that $\gamma$ is a.s.\ H\"older continuous of order $\alpha-1/p$ on $I$. We can now finish the proof by choosing $\alpha$  close to $1/d$ and $p$ arbitrarily big.
\end{proof}

\begin{Remark}
	We actually proved that, for any $\alpha<1/d$, the $\alpha$-H\"older norm of $\gamma$ restricted to a bounded interval has finite $p$-moment for any $p>0$.
\end{Remark}

\begin{proof}[Proof of Theorem \ref{Theorem-Mckean*}]
	Since $\gamma$ is a.s.\ locally $\alpha$-H\"older continuous for any $\alpha<1/d$,  we get the upper bound of (\ref{Mckean}): $\dim_H(\gamma(A))\le d \cdot \dim_H(A)$ (cf.\ \cite[Remark 4.12]{BM}).
	
	For the lower bound, we use a corollary of Frostman's Lemma (cf.\ \cite{Frostman}). 
	Let ${\cal P}(E)$ denote the space of probability measures on a metric space $(E,\rho)$. For $\mu\in{\cal P}(E)$ and $\alpha\ge 0$, the $\alpha$-energy of $\mu$ is
	$$I_\alpha(\mu):=\int_E\!\int_E \frac{\mu(dx)\mu(dy)}{\rho(x,y)^\alpha}.$$
	Frostman's lemma implies that (cf.\ \cite[Theorem 4.36]{BM}) for any closed set $A\subset \R^n$,
	\BGE \dim_H(A)=\sup\{\alpha:\exists\mu\in{\cal P}(A)\mbox { with }I_\alpha(\mu)<\infty\}.\label{Frostman*}\EDE
	
	We now prove the lower bound of (\ref{Mckean}). Let $A\subset\R$ be closed. By replacing $A$ with $A\cap [-n,n]$, we may assume that $A$ is compact. If $\dim_H(A)=0$, the statement of lower bound is trivial. Suppose $\dim_H(A)>0$. Let $\alpha\in(0,\dim_H(A))$. By (\ref{Frostman*}), there is a $\mu\in{\cal P}(A)$ such that $I_\alpha(\mu)<\infty$.
	Now $\gamma(A)$ is a compact subset of $\C$, and $\gamma_*(\mu)\in{\cal P}(\gamma(A))$. We compute that
	$$\EE[I_{d\alpha}(\gamma_*(\mu))]=\int_A\!\int_A \EE[|\gamma(s)-\gamma(t)|^{-d\alpha}] \mu(ds)\mu(dt)=I_\alpha(\mu)\EE[|\gamma(1)|^{-d\alpha}]<\infty,$$
	using Proposition \ref{stationary*}, Lemma \ref{integrable*} and the fact that $-d\alpha>-d$, which follows from $\alpha<\dim_H(A)\le 1$. Thus, we have a.s.\ $I_{d\alpha}(\gamma_*(\mu))<\infty$. From (\ref{Frostman*}), we then get a.s.\ $\dim_H(\gamma(A))\ge d\alpha$. By choosing $\alpha$ arbitrarily close to $\dim_H(A)$, we finish the proof of the lower bound.
\end{proof}

The rest of the paper is devoted to proving Lemma \ref{integrable*} and Theorem \ref{not-Holder}. We review some preliminary facts on two-sided whole-plane SLE and Minkowski content in Section \ref{Preliminary}, and then prove Lemma \ref{integrable*} and Theorem \ref{not-Holder} in Section \ref{proof-key-lem} assuming two lemmas about crossing estimates for radial SLE$_\kappa(2)$ curves. In Section \ref{Section-crossing}, we prove these crossing estimate lemmas. At the end,  we discuss  sssi SLE$_\kappa$ curves for $\kappa\ge 8$ and the Hausdorff measure of SLE$_\kappa$ curves for $\kappa<8$.

\section*{Acknowledgments}
The author thanks Greg Lawler and Nina Holden for inspiring discussions and valuable comments, acknowledges the support from the National Science Foundation (DMS-1056840) and from the Simons Foundation (\#396973), and thanks Columbia University and KIAS, where part of this work was carried out during two conferences held by them. The author also thanks the anonymous referees, whose comments improved the quality of the paper.

\section{Preliminary} \label{Preliminary}

\subsection{Two-sided whole-plane SLE} \label{Section-two-sided}

A two-sided whole-plane SLE$_\kappa$ curve ($\kappa\in(0,8)$) is   composed of two arms: the first arm  is a whole-plane SLE$_\kappa(2)$ curve from a marked point in   $\ha\C$, say $a$, to another marked point in $\ha\C$, say $b$; and given the first arm, the second arm is a chordal SLE$_\kappa$ curve from $b$ to $a$ in one complement domain of the first arm. 


A whole-plane SLE$_\kappa(2)$ curve $\gamma$ grows in $\ha\C$ from one (interior) point, say $a$, to another point, say $b$, with the force point located at the initial point $a$. It is  related with  radial SLE$_\kappa(2)$ processes by the domain Markov property: if $\tau$ is a stopping time for $\gamma$, which happens after $\gamma$ leaves $a$ and before $\gamma$ reaches $b$, then conditionally on the part of $\gamma$ before $\tau$, the rest of the $\gamma$ is a radial SLE$_\kappa(2)$ curve in a complement domain of the explored portion of $\gamma$ at time $\tau$ from $\gamma(\tau)$ to $b$ with the force point located at $a$. Here $a$ lies on the boundary of that domain, and determines a prime end. The above also holds with $2$ replaced by some $\rho\ge \frac \kappa 2-2$.

We now recall the definition of a radial SLE$_\kappa(2)$ curve. Let $\D=\{z\in\C:|z|<1\}$, $\D^*=\ha\C\sem \lin{\D}$, and  $\TT=\pa\D^*=\{|z|=1\}$. 
Let $a,b$ be distinct points on $\TT$. Let $x,y\in\R$ be such that $a=e^{ix}$, $b=e^{iy}$ and $x-y\in(0,2\pi)$. Let $(B_t)$ be a standard Brownian motion. Let $(\lambda_t)$ and $(q_t)$ be the solution of the SDE:
\BGE \left\{ \begin{array}{ll} d\lambda_t=\sqrt{\kappa}dB_t +\cot((\lambda_t-q_t)/2)dt, & \lambda_0=x\\
	dq_t=\cot((q_t-\lambda_t)/2)dt, &q_0=y.
\end{array}\right.\label{SDE-lambda-q}\EDE

Let $g_t$, $t\ge 0$, be the solution of the following radial Loewner equation:
$$\pa_t g_t(z)=g_t(z)\cdot \frac{e^{i\lambda_t}+g_t(z)}{e^{i\lambda_t}-g_t(z)},\quad g_0(z)=z.$$
For each $t\ge 0$, let $K_t$ denote the set of $z\in\D^*$ such that the solution $s\mapsto g_s(z)$ blows up before or at time $t$. Then $g_t$ maps $\D^*\sem K_t$ conformally onto $\D^*$, fixes $\infty$, and satisfies $g_t(z)/z\to e^{-t}$ as $z\to\infty$. It turns out that, for each $t\ge 0$, $g_t^{-1}$ extends continuously to $\lin{\D^*}$, and $\gamma(t):=g_t^{-1}(e^{i\lambda_t})$, $t\ge 0$, is a continuous curve in $\lin{\D^*}$ that starts from $e^{i\lambda_0}=e^{ix}=a\in\TT$, and tends to $\infty$ as $t\to \infty$. Moreover, for each $t\ge 0$, $\D^*\sem K_t$ is the unbounded component of $\D^*\sem \gamma([0,t])$.

Such $\gamma$ is called a radial SLE$_\kappa(2)$ curve in $\D^*$ from $a$ to $\infty$ with a (value $2$) force point at $b$; $\lambda_t$ and $q_t$ are called the driving process and force point process for $\gamma$, respectively; and   $g_t$ are called the radial Loewner maps for $\gamma$.

If $D$ is a simply connected domain with two distinct marked prime ends $a,b$ and one marked interior point $z_0$, there exist $a'\ne b'\in\TT$ and a conformal map $f$ from $(\D^*;a',b';\infty)$ onto $(D;a,b;z_0)$. Then the $f$-image of a radial SLE$_\kappa(2)$ curve in $\D^*$ from $a'$ to $\infty$ with a force point at $b'$ is called a radial SLE$_\kappa(2)$ curve in $D$ from $a$ to $z_0$ with a force point at $b$.

The radial SLE$_\kappa(2)$ curve $\gamma$ in $D$ satisfies the following domain Markov property: if $T$ is a stopping time for $\gamma$ that happens before $z_0$ is reached, then conditional on the part of $\gamma$ before $T$, the part of $\gamma$ after $T$ is a  radial SLE$_\kappa(2)$ curve from $\gamma(T)$ to $z_0$ with a force point at $b$ in the complement domain of $\gamma([0,T])$ in $D$ that contains $z_0$.

The above radial SLE$_\kappa(2)$ curve $\gamma$ also acts  the role of the first arm of
a two-sided radial SLE$_\kappa$ curve in $D$ from $a$ to $b$ through $z_0$, which is composed of two arms. Conditional on the first arm, the second arm is a chordal SLE$_\kappa$ curve from $z_0$ to $b$ in the complement domain of the first arm, whose boundary contains $b$.

Throughout, we  use   $\mu^\#_{D;a\to b}$ to denote the law of a chordal SLE$_\kappa$ curve in $D$ from $a$ to $b$;  use $\nu^\#_{D;a\to z_0; b}$ to denote the law of a  radial SLE$_\kappa(2)$ curve in $D$ from $a$ to $z_0$ with a  force point at $b$; and use $\nu^\#_{D;a\to z_0\to b}$ to denote the law of a two-sided radial SLE$_\kappa$ curve in $D$ from $a$ to $b$ passing through $z_0$, all modulo time parametrization.

The Green's function for the chordal SLE$_\kappa$ curve in $D$ from $a$ to $b$ is the function $G_{D;a,b}(z)$ defined by $$G_{D;a, b}(z_0)=\lim_{r\downarrow 0} r^{d-2} \mu^\#_{D;a\to b} \{\gamma:\dist(\gamma, z_0)<r\},\quad z_0\in D.$$
The limit is known to exist, and satisfies conformal covariance, i.e., if $f$ maps $(D_1;a_1,b_1)$ conformally onto $(D_2;a_2,b_2)$, then
\BGE G_{D_1;a_1, b_1}(z)=|f'(z)|^{2-d} G_{D_2;a_2,b_2}(f(z)),\quad z\in D_1.\label{Green-covariance}\EDE

\subsection{Minkowski content measure} \label{Minkowski-Section}
In this subsection, we review the notion of Minkowski content, Minkowski content measure, and cite some propositions from \cite[Sections 2.3 and 2.4]{loop}.

Throughout, the Minkowski content will always be $d$-dimensional. Since $d=1+\frac{\kappa}8$ is fixed, we will omit the word ``$d$-dimensional''. Let $\mA$ and $\mA^2$ denote the $1$ and $2$-dimensional Lebesgue measures, respectively.
Let $S\subset\C$ be a closed set. The Minkowski content of $S$ is defined to be
$ \Cont(S)=\lim_{r\downarrow 0} r^{d-2}{ \mA^2(\bigcup_{z\in S} B(z;r))}$, where $B(z;r)$ is the open disc with radius $r$ centered at $z$,
provided that the limit exists.

\begin{Definition}
	Let $S,U\subset\C$. Suppose $\M$ is a measure supported by $S\cap U$ such that
	for every compact set $K\subset U$, $\Cont(K\cap S)= \M (K)<\infty$.
	Then we say that $\M$ is the  Minkowski content measure on $S$ in $U$, or $S$ possesses  Minkowski content measure in $U$, and denote the measure by  $\M_{S;U}$. If $U=\C$, we may omit the phrase ``in $U$'', and write the measure as $\M_S$.
\end{Definition}

\begin{Proposition} \cite[Lemma 2.6]{loop}
	Suppose that $S$ possesses  Minkowski content measure $\M_{S;U}$ in an open set $U\subset\C$. Suppose $f$ is a conformal map defined on $U$ such that $f(U)\subset\C$.
	Then the Minkowski content measure of $f(S\cap U)$ in $f(U)$ exists, which is absolutely continuous w.r.t.\  $f_*(\M_{S;U})$, and the Radon-Nikodym derivative is $|f'(f^{-1}(\cdot))|^{d}$. \label{conformal-content}
\end{Proposition}


 It was proved in \cite{LR} that the natural parametrization function $\Theta_t$ of a chordal SLE$_\kappa$ curve $\gamma$ (under capacity parametrization) in $\HH$ from $0$ to $\infty$ at any time $t$ agrees with the Minkowski content of $\gamma([0,t])$. Using their result, one can easily show that such curve $\gamma$ possesses Minkowski content measure in $\HH$, which is the pushforward measure of $d\Theta$ under the function $\gamma$. The existence of Minkowski content measure extends to any simply connected domain (\cite[Remark 2.7]{loop}), although the natural parametrization (in the usual sense) does not always exist  because the Minkowski content of $\gamma([0,t])$ may be infinite due to the roughness of the domain boundary. Moreover,  the following proposition holds.

\begin{Proposition}\cite[Proposition 2.8]{loop}
	Let $D$ be a simply connected domain with two distinct prime ends $a$ and $b$. Then
	$$ \mu^\#_{D;a\to b}(d\gamma)\otimes {\cal M}_{\gamma;D}(dz)=\nu^\#_{D;a\to z\to b}(d\gamma)\overleftarrow{\otimes} (  G_{D;a, b}\cdot \mA^2)(dz).$$
	\label{decomposition-Thm}
\end{Proposition}

This proposition first appeared in \cite[Corollary 4.3]{decomposition} (in a slightly different form). It means that, there are two ways to sample a random curve-point pair $(\gamma,z)$ according to the same measure. Method $1$: first sample a chordal SLE$_\kappa$ curve $\gamma$ in $D$ from $a$ to $b$, and then sample a point $z$ on $\gamma$ according to the Minkowski content measure of $\gamma$ in $D$. Method $2$: first sample a point $z$ according to the density $G_{D;a, b}$, and then sample a two-sided radial SLE$_\kappa$ curve $\gamma$ in $D$ from $a$ to $b$ passing through $z$. The reader is referred to \cite[Section 2.1]{loop} for the meaning of the symbols $\otimes$ and $\overleftarrow{\otimes}$ for the products between a measure and a kernel.

\section{Finite Moments} \label{proof-key-lem}
We prove Lemma \ref{integrable*} and Theorem \ref{not-Holder} in this section. We divide the proof of Lemma \ref{integrable*} into two subsections according to the sign of the exponent $c$.

Let $\gamma$ be as in Proposition \ref{stationary*}. Let $\gamma_+=\gamma|_{[0,\infty)}$ and $\gamma_-=\gamma|_{(-\infty,0]}$ be the two arms of $\gamma$, both connecting $0$ with $\infty$. Let $\ha\C_{\gamma_-}$ be the unbounded connected component of $\C\sem \gamma_-$. Recall that $\gamma_-$ is a whole-plane SLE$_\kappa(2)$ curve from $\infty$ to $0$; and given $\gamma_-$, $\gamma_+$ is a chordal SLE$_\kappa$ curve from $0$ to $\infty$ in $\ha\C_{\gamma_-}$.  Let $\gamma^R_-(t)=\gamma(-t)$, $0\le t<\infty$, be the time-reversal of $\gamma_-$. From the reversibility of whole-plane SLE$_\kappa(2)$ (\cite[Theorem 1.20]{MS4}), we know that $\gamma^R_-$ is a whole-plane SLE$_\kappa(2)$ curve from $0$ to $\infty$.

For $0\le t<\infty$, let $D_t$ denote the complement domain of $\gamma^R_-([0,t])$ in $\ha\C$ that contains $\infty$. Then $K_t:=\ha\C\sem D_t$ is an interior hull in $\C$.
An interior hull in $\C$ is a nonempty connected compact set $K\subset\C$ such that $\C\sem K$ is connected. It is called nondegenerate if it contains more than one point. For a nondegenerate interior hull $K$, there exists a unique conformal map $g_K$ from $\ha\C\sem K$ onto $\D^*$ such that $g_K(\infty)=\infty$ and $g_K'(\infty):=\lim_{z\to\infty} z/g_K(z)>0$. The   capacity of $K$ is defined to be $\ccap(K)=\log(g_K'(\infty))$. If $K$ is degenerate, i.e., a singlet, we define $\ccap(K)=-\infty$. From Schwarz Lemma and Koebe $1/4$ Theorem, we know that $e^{\ccap(K)}\le\diam(K)\le 4e^{\ccap(K)}$ for any interior hull $K$ in $\C$.

Let $(\F_t)_{t\ge 0}$ denote the filtration generated by $(\gamma^R_-(t))_{t\ge 0}$.
If $T\in(0,\infty)$ is an $(\F_t)$-stopping time, then conditional on $\F_{T}$, the part of $\gamma^R_-$ after $T$ is an SLE$_\kappa(2)$ curve in $D_{T}$ from $\gamma^R_-(T)$ to $\infty$ with a force point at $0$. Let $u(t)=\ccap(K_t)$. Then $u$ is an increasing homeomorphism from $[0,\infty)$ onto $[-\infty,\infty)$ and for every $r\in\R$, $u^{-1}(r)$ is an $(\F_t)$-stopping time.

For $R,r>0$, let $\tau_R$ denote the hitting time (i.e., the first visiting time) of the circle $\{|z|=R\}$; and let $\tau_r^R$ denote the last visiting time of the circle  $\{|z|=r\}$ before $\tau_R$. Here we set $\tau_R$ to be $\infty$ when the former time does not exist, and set $\tau_r^R$ to be $\tau_R$ when the latter time does not exist.  Note that $\tau_R$ is a stopping time, but $\tau_r^R$ is not in general. For any $R>0$, since $K_{\tau_R}\subset \{|z|\le R\}$ and $\diam(K_{\tau_R})\ge R$, we have $\log(R)\ge \ccap(K_{\tau_R})\ge \log(R/4)$, and so $\tau_R\le u^{-1}(\log R)\le \tau_{4R}$.

\subsection{Negative power}
\begin{proof}[Proof of Lemma \ref{integrable*} for $c\le 0$.]
Let $f_c(z)={\bf 1}_{\{|z|\le 1\}}|z|^c$. We will show that \BGE \EE\Big[\int_0^\infty f_c(\gamma(t)) dt\Big]<\infty.\label{int-infty*}\EDE
From the scaling property of $\gamma$, we get
$$\EE\Big[\int_0^\infty f_c(\gamma(t)) dt\Big]\ge \int_0^1 \EE[f_c(\gamma(t))]dt
=\int_0^1 \EE[f_c(t^{1/d}\gamma(1))]dt$$$$=\int_0^1 \EE[{\bf 1}_{\{|\gamma(1)|\le t^{-1/d}\}}t^{c/d}|\gamma(1)|^{c}]dt\ge \EE[f_c (\gamma(1))] \int_0^1 t^{c/d}dt.$$
Thus, (\ref{int-infty*}) implies that $\EE[f_c(\gamma(1))]<\infty$, which then implies that $\EE[|\gamma(1)|^c]<\infty$.

Since the whole-plane SLE$_\kappa(2)$ curve $\gamma(t)$, $-\infty<t<0$, visits its target $\gamma(0)=0$ only at the end,  we have a.s.\   $\gamma(0)\not\in \gamma_-= \gamma((-\infty,0))$. Since $\gamma$ has stationary increment, we see that for any $t>0$, a.s.\ $\gamma(t)\not\in \gamma_-$. By Fubini Theorem, the Lebesgue measure of the set of $t>0$ such that $\gamma(t)\in\gamma_-$ is a.s.\ zero. Since $\gamma$ is parametrized by Minkowski content measure ${\cal M}_\gamma$, i.e., ${\cal M}_\gamma$ is the pushforward measure of the Lebesgue measure under the curve function $\gamma$, we then conclude that, a.s.\ the set $\gamma_+\cap \gamma_-$ has Minkowski content measure zero. We know that $\gamma_+$ is contained in the closure of $\ha\C_{\gamma_-}$, and when $\gamma_+$ touches the boundary of $\ha\C_{\gamma_-}$, it intersects $\gamma_-$. Thus, the set of points on $\gamma_+$ that is not strictly contained in $\ha\C_{\gamma_-}$ has zero Minkowski content measure almost surely. So we have a.s.\ ${\cal M}_{\gamma_+}={\cal M}_{\gamma_+;\ha\C_{\gamma_-}}$.

Since $\gamma_+$ is a chordal SLE$_\kappa$ curve from $0$ to $\infty$ in $\ha\C_{\gamma_-}$, from Proposition \ref{decomposition-Thm} and that a.s.\ ${\cal M}_{\gamma_+}={\cal M}_{\gamma_+;\ha\C_{\gamma_-}}$, we get
\BGE \mu^\#_{\ha\C_{\gamma_-};0\to\infty}(d\gamma_+)\otimes {\cal M}_{\gamma_+}(dz)=\nu^\#_{\ha\C_{\gamma_-};0\to z\to\infty}(d\gamma_+)\overleftarrow{\otimes} (G_{\ha\C_{\gamma_-};0,\infty}\cdot \mA^2)(dz).\label{decomposition*}\EDE
Since $\gamma$ is parametrized by Minkowski content, we have
$ \int_0^\infty f_c(\gamma(t))dt = \int_{\C} f_c(z){\cal M}_{\gamma_+}(dz)$.
Since the law of $\gamma_+$ given $\gamma_-$ is $\mu^\#_{\ha\C_{\gamma_-};0\to\infty}$, using (\ref{decomposition*}) we get
$$\EE\Big[\int_{\C} f_c(z){\cal M}_{\gamma_+}(dz)\Big|\gamma_- \Big]=\int f_c(z) [\mu^\#_{\ha\C_{\gamma_-};0\to\infty}(d\gamma_+)\otimes {\cal M}_{\gamma_+}](dz)$$
$$=\int f_c(z) [\nu^\#_{\ha\C_{\gamma_-};0\to z\to\infty} \overleftarrow{\otimes} (G_{\ha\C_{\gamma_-};0,\infty}\cdot \mA^2)](dz)=\int f_c(z) G_{\ha\C_{\gamma_-};0,\infty}(z)\mA^2(dz).$$
Thus, we have
\BGE \EE\Big[\int_0^\infty f_c(\gamma(t))dt\Big]=\int_{\C} f_c(z) \EE[G_{\ha\C_{\gamma_-};0,\infty}(z)]\mA^2(dz).\label{EG1}\EDE
We will prove that
\BGE \EE[G_{\ha\C_{\gamma_-};0,\infty}(z_1)]<\infty,\quad \forall z_1\in\C\sem \{0\}.\label{EG*}\EDE
If this holds, by scaling and rotation invariance of $\gamma_-$ and conformal covariance of chordal SLE$_\kappa$ Green's function, we see that $\EE[G_{\ha\C_{\gamma_-};0,\infty}(z)]=C|z|^{d-2}$, $z\in\C\sem\{0\}$, for some constant $C=C_\kappa\in(0,\infty)$.
Combining it with (\ref{EG1}), we get  $\EE[\int_0^\infty f_c(\gamma(t))dt]=C\int_{\{|z|\le 1\}} |z|^{c+d-2}\mA^2(dz)$. Using $c>-d$, we get (\ref{int-infty*}) and finish the proof.

Now it remains to prove (\ref{EG*}).
We claim that there are constants $C,R\in(1,\infty)$ such that the following holds. Let $D\subsetneqq \C$ be a simply connected domain with distinct prime ends $a,b$. Let $z_0,z_1$ be distinct points in $D$ such that $\dist(z_0,\pa D)\ge R|z_1-z_0|$.   Let $\beta$ have the law $\nu^\#_{D;a\to z_0;b}$, and   $D_\beta$ be the complement domain of $\beta$ in $D$ that has $b$ as a prime end. Then \BGE \EE[G_{D_\beta;z_0,b}(z_1)]\le C|z_1-z_0|^{d-2}\label{2-sided-Green}\EDE

If the claim holds true, we can finish the proof as follows. Fix $z_1\in\C\sem \{0\}$. Let $T=\tau_{R|z_1|}$. Let $D^T$ be the complement domain of $\gamma_-((-\infty,T])$ that contains $0$. Conditional on the part of $\gamma_-$ before $T$, the remaining part of $\gamma_-$, denoted by $\beta_0$, is a  radial SLE$_\kappa(2)$ curve in  $D^T$ from $\gamma(T)$ to $0$ with a force point at $\infty$. Then we have $(D^T)_{\beta_0}=\ha\C_{\gamma_-}$, and $\dist(0,\pa D^T)=R|0-z_1|$.  Applying the claim to $D=D^T$,  $a=\gamma(T)$, $b=\infty$, $z_0=0$, and $\beta=\beta_0$, we get (\ref{EG*}) because
$$\EE[G_{\ha\C_{\gamma_-};0,\infty}(z_1)]=\EE[\EE[G_{\ha\C_{\gamma_-};0,\infty}(z_1)|\gamma_{(-\infty,T]}]]=\EE[G_{(D^T)_{\beta_0};0,\infty}(z_1)]\le C |z_1|^{d-2}<\infty.$$

To prove the claim, by Koebe distortion Theorem and the conformal covariance of chordal SLE$_\kappa$ Green's function, it suffices to work on the case that $D=\HH:=\{z\in\C:\Imm z>0\}$, $a=0$ and $b=\infty$. 
To estimate $\EE[G_{\HH_\beta;z_0,b}(z_1)]$, we use the formula
\BGE \overrightarrow G_{\HH;0,\infty}(z_0,z_1)=G_{\HH;0,\infty}(z_0) \EE[G_{\HH_\beta;z_0,\infty}(z_1)],\label{LW*}\EDE
where $\overrightarrow G_{\HH;0,\infty}(z_0,z_1)$ is the {\it ordered} two-point Green's function for the chordal SLE$_\kappa$ curve in $\HH$ from $0$ to $\infty$, i.e.,
$$\overrightarrow G_{\HH;0,\infty}(z_0,z_1):=\lim_{r_0,r_1\downarrow 0} r_0^{d-2} r_1^{d-2} \mu^\#_{\HH;0\to \infty}[\tau^{z_0}_{r_0}<\tau^{z_1}_{r_1}<\infty],$$
where $\tau^{z_j}_{r_j}$ is the hitting time of $\{|z-z_j|=r_j\}$.

Formula (\ref{LW*}) was first derived in \cite[Theorem 1]{LW}, where the Euclidean distance is replaced with conformal radius. The Euclidean distance version of (\ref{LW*}) was later derived in \cite[Proposition 4.5]{LR}. A sharp bound for the {\it unordered} two-point Green's function:
$$ G_{\HH;0,\infty}(z_0,z_1):=\lim_{r_0,r_1\downarrow 0} r_0^{d-2} r_1^{d-2}  \mu^\#_{\HH;0\to \infty}[\tau^{z_0}_{r_0},\tau^{z_1}_{r_1}<\infty]$$
is obtained in \cite{LR2}. Combining it with (\ref{LW*}) and the exact formula for $G_{\HH;0,\infty}(z_0)$, we then get an upper bound for $\EE[G_{\HH_\beta;z_0,\infty}(z_1)]$.

Now we show the computation. For convenience, we use an upper bound of $G_{\HH;0,\infty}(z_0,z_1)$ in the form derived in \cite[Remark 2 after Theorem 1.2]{RZ}:
\BGE G_{\HH;0,\infty}(z_0,z_1)\le C_2 \frac{y_0^{\alpha-(2-d)}}{P_{y_0}(l_0)}\cdot \frac{y_1^{\alpha-(2-d)}}{P_{y_1}(l_1)},\label{2-point*}\EDE
where $C_2\in(0,\infty)$ is a constant, $\alpha=\frac 8\kappa-1$, $y_j=\Imm z_j$, $j=0,1$, $l_0=|z_0|$, $l_1=|z_1|\wedge |z_1-z_0|$, and
$P_y(x)=x^\alpha$ if $x\ge y$; $P_y(x)=y^{\alpha-(2-d)}x^{2-d}$ if $0\le x\le y$. From \cite[Theorem 4.2]{LR} we know that there is a constant $C_1\in(0,\infty)$ such that
\BGE G_{\HH;0,\infty}(z_0)=C_1 y_0^{\alpha-(2-d)}|z|^{-\alpha}=C_1 \frac{y_0^{\alpha-(2-d)}}{P_{y_0}(l_0)}.\label{1-point*}\EDE
Since $\overrightarrow G_{\HH;0,\infty}(z_0,z_1)\le G_{\HH;0,\infty}(z_0,z_1)$, combining (\ref{LW*},\ref{2-point*},\ref{1-point*}) we get
$$\EE[G_{\HH_\beta;z_0,\infty}(z_1)]\le \frac{C_2}{C_1}\cdot \frac{y_1^{\alpha-(2-d)}}{P_{y_1}(l_1)}.$$
Suppose $y_0=\dist(z_0,\R)\ge R|z_1-z_0|$ with $R=2$. Then $|z_1-z_0|\le y_0-|z_1-z_0|\le y_1\le |z_1|$. So $l_1=|z_1-z_0|\le y_1$, and the above displayed formula implies (\ref{2-sided-Green}) because $P_{y_1}(l_1)= y_1^{\alpha-(2-d)}|z_1-z_0|^{2-d}$. So the proof is complete.
\end{proof}

\subsection{Positive power}

Let $\cal D$ denote the set of simply connected domains $D\subsetneqq \ha\C$ such that $\ha\C\sem D$ is a nondegenerate interior hull in $\C$ containing $0$. Note that $\infty\in D$ for $D\in\cal D$.

\begin{Lemma}
	There exist $R>r>4$ such that for any $p>0$   there exists $l>0$ such that for any $D\in\cal D$ with distinct prime ends $a,b$,
	$$\nu^\#_{D;a\to \infty;b}(\{\beta:\Cont(\beta([\tau^{Re^{\ccap(D^c)}}_{re^{\ccap(D^c)}},{\tau_{Re^{\ccap(D^c)}}}]))<l*e^{d*\ccap(D^c)}\})<p.$$
	\label{plrR2*}
\end{Lemma}

We now use Lemma \ref{plrR2*} to prove Lemma \ref{integrable*} for positive powers, and postpone the proof of Lemma \ref{plrR2*} to the next section.

\begin{proof}[Proof of Lemma \ref{integrable*} for $c> 0$.]
We work on the whole-plane SLE$_\kappa(2)$ curve $\gamma^R_-$ from $0$ to $\infty$.
 By the stationary increments of $\gamma$, $\diam(\gamma([0,1]))$ has the same law as $\diam(\gamma^R_-([0,1]))$. So it suffices to show that $\diam(\gamma^R_-([0,1]))$ has finite moments of any positive order.
	
	Let $R>r\in(4,\infty)$ be given by Lemma \ref{plrR2*}. Fix $p>0$. Let $l=l(p,r,R)>0$ be given by Lemma \ref{plrR2*}. Let  $s_n=\log(R^n/l^{1/d})$ for $n\in\N\cup\{0\}$. 
	Applying Lemma \ref{plrR2*} to $D=D_{u^{-1}(s_k)}$,  $a=\gamma^R_-(u^{-1}(s_k))$ and $b=0$, we get
	$$\PP[u^{-1}(s_{k+1})-u^{-1}(s_k)<l*e^{d*s_k}|\F_{u^{-1}(s_k)}]<p,\quad k=0,1,2,\dots.$$
	Here we use the facts that given $\F_{u^{-1}(s_k)}$, $\gamma^R_-({u^{-1}(s_k)}+t)$, $t\ge 0$, has the law of $\nu^\#_{D;a\to z_0;b}$; $u^{-1}(s_{k+1})-u^{-1}(s_k)=\Cont(\gamma^R_-([u^{-1}(s_k),u^{-1}(s_{k+1})]))$;  $u^{-1}(s_{k+1})\ge \tau_{e^{s_k+1}}=\tau_{Re^{s_k}}$, $u^{-1}(s_k)\le \tau_{4 e^{s_k}}< \tau_{re^{s_k}}\le \tau_{re^{s_k}}^{Re^{s_k}}$; and $s_k=\ccap(D^c)$.  Since $u^{-1}(s_k)-u^{-1}(s_{k-1})$ is $\F_{u^{-1}(s_k)}$-measurable and $l*e^{d*s_k}=R^{d*k}\ge 1$, we get
	$$\PP[s_n<u(1)]=\PP[u^{-1}(s_n)< 1]\le \PP\Big[\bigcap_{k=1}^n \{u^{-1}(s_k)-u^{-1}(s_{k-1})< l* e^{d*s_k}\}\Big]<p^n,\quad n\in\N.$$
	Since $\diam(K_1)\ge \diam(\gamma^R_-([0,1]))$ and $u(1) =\ccap(K_1)\ge \log(\diam(K_1)/4)$, we see that $\diam(\gamma^R_-([0,1]))>4R^n/l^{1/d}$ implies that $u(1)>s_n$. Thus,
	$\PP[\diam(\gamma^R_-([0,1]))>4R^n/l^{1/d}]<p^n$. This shows that $\EE[\diam(\gamma^R_-([0,1]))^c]<\infty$ if $c<\log(1/p)/\log(R)$. Since $R$ is fixed, taking $p$ arbitrarily small completes the proof.
\end{proof} 

\subsection{Critical exponent}

\begin{Lemma} Suppose $\kappa\in(0,4]$.
	There exist $R>r>4$ such that for any $l>0$ there exists $p>0$ such that for any $D\in {\cal D}$ with distinct prime ends $a,b$,
	$$\nu^\#_{D;a\to \infty;b}(\{\beta:\Cont(\beta([\tau^{Re^{\ccap(D^c)}}_{re^{\ccap(D^c)}},\tau_{Re^{\ccap(D^c)}}]))<l*e^{d*\ccap(D^c)}\})>p.$$
	\label{plrR3*}
\end{Lemma}

We now use Lemma \ref{plrR3*} to prove Theorem \ref{not-Holder}, and
 postpone the proof of Lemma \ref{plrR3*} to the next section.

\begin{proof} [Proof of Theorem \ref{not-Holder}.]
	By self-similarity and stationarity of increments, it suffices to prove that $\gamma^R_-$ is a.s.\ not H\"older continuous on $[0,\tau_1]$.   Let $R>r>4$ be given by Lemma \ref{plrR3*}. Fix $l>0$. Let $p=p(l,r,R)$ be given by Lemma \ref{plrR3*}.
	
	Let $T_n=\tau_{R^{-n}}$, $n\in\N\cup\{0\}$. Let $E_n$ denote the event that $\tau_{Re^{u(T_n)}}-\tau^{R e^{u(T_n)}}_{r e^{u(T_n)}}\le l* e^{d* u(T_n)}$.  Applying Lemma \ref{plrR3*} to $D=D_{T_n}$, $a=\gamma^R_-(T_n)$ and $b=0$, and using the fact that the law of $\gamma^R_-(T_n+t)$, $t\ge 0$, given $\F_{T_n}$, is $\nu^\#_{D;a\to \infty;b}$,
	we get $\PP[E_n^c|\F_{T_n}]<1-p$. From $u(T_n)\le \log(R^{-n})$, we get $Re^{u(T_n)}\le R^{1-n}$, and so $E_n$ is measurable w.r.t.\ $\F_{T_{n-1}}$. Thus,
	$\PP[\bigcup_{n=1}^\infty E_n]=1-\PP[\bigcap_{n=1}^\infty E_n^c]\ge 1-(1-p)^\infty=1$. If $E_n$ happens for some $n\in\N$, then $\gamma^R_-([0,\tau_1])$ contains a subcurve crossing the annulus $\{re^{u(T_n)}<|z|<R e^{u(T_n)}\}$ with time duration no more than $l*e^{d* u(T_n)}$, which implies that
	$$\sup_{0\le s<t\le \tau_1}\frac{|\gamma^R_-(t)-\gamma^R_-(s)|}{(t-s)^{1/d}}\ge \frac{R e^{u(T_n)}-r e^{u(T_n)}}{(l*e^{d* u(T_n)})^{1/d}}=\frac{R-r}{l^{1/d}}.$$
	Thus, a.s.\ $\sup_{0\le s<t\le \tau_1}\frac{|\gamma^R_-(t)-\gamma^R_-(s)|}{(t-s)^{1/d}}\ge \frac{R-r}{l^{1/d}}$. Since $R>r>0$ are fixed, taking $l>0$ arbitrarily small completes the proof.
\end{proof}

\section{Crossing Estimates} \label{Section-crossing}
The purpose of the last section is to prove Lemmas \ref{plrR2*} and \ref{plrR3*}. First, we use Koebe distortion theorem to derive the following proposition. For $a>0$, let ${\cal U}_a$ denote the set of univalent analytic functions $f$ defined on $\D$ such that $f(0)=0$ and $|f'(0)|=a$. 

\begin{Proposition}
	\begin{enumerate}
		\item [(i)] For every $a>0$ and $R_2>R_1\in(0,1)$ such that $R_2< 1/6$, there exist $r_2>r_1>0$ such that $r_2< \frac{3}{2} a R_2$, and
		\BGE \{r_1\le |z|\le r_2\}\supset f(\{R_1\le |z|\le R_2\}),\quad \forall f\in{\cal U}_a.\label{i3}\EDE
		\item[(ii)] For every $a>0$ and $R_2>R_1\in (0,1)$ such that $R_2/R_1\ge 6$,   there exist $r_2>r_1>0$ satisfying $r_2<a/4$ and $r_2/r_1> (R_2/R_1)/6$ such that
		\BGE\{r_1\le|z|\le r_2\}\subset f(\{R_1\le |z|\le R_2\}),\quad \forall f\in{\cal U}_a.\label{i1}\EDE
		\item[(iii)]	For every $a>0$ and $r_2>r_1>0$ such that $r_2<a/4$ and $r_2/r_1\ge 12$,  there exist $R_2>R_1\in (0,1)$ such that $R_2/R_1>(r_2/r_1)/12$, and
		\BGE \{R_1\le |z|\le R_2\}\subset f^{-1}(\{r_1\le |z|\le r_2\}),\quad \forall f\in{\cal U}_a.\label{i2}\EDE
	\end{enumerate} \label{Lem-Koebe}
\end{Proposition}
\begin{proof}
	By scaling we may assume that $a=1$. By Koebe distortion theorem, for any $f\in{\cal U}_1$,
	\BGE \frac{|z|}{(1+|z|)^2}\le |f(z)|\le \frac{|z|}{(1-|z|)^2},\quad z\in\D.\label{eqn-Koebe}\EDE
	
	(i) Suppose $1/6\ge R_2>R_1>0$. Let $r_2=\frac{R_2}{(1-R_2)^2}$ and $r_1=\frac{R_1}{(1+R_1)^2}$. Then $r_2
	> R_2 >R_1>r_1$, $r_2\le \frac{R_2}{(1-1/6)^2}<\frac {3}{2}R_2$, and we get (\ref{i3})  using (\ref{eqn-Koebe}).
	
	(ii) Suppose $R_1,R_2\in(0,1)$ and $R_2/R_1\ge 6$.  Let $r_2=\frac{R_2}{(1+R_2)^2}$ and $r_1=\frac{R_1}{(1-R_1)^2}$. Then from $R_1\le 1/6$ and $R_2\le 1$, we get $\frac{r_2}{r_1}\ge \frac{(1-1/6)^2}{(1+1)^2}\cdot\frac{R_2}{R_1}>\frac 16\cdot\frac{R_2}{R_1}$. Since $R_2/R_1\ge 6$, we get $r_2>r_1$. Formula (\ref{i1}) now follows from (\ref{eqn-Koebe}).
	
	(iii) Suppose $r_1,r_2>0$, $r_2<1/4$, and $r_2/r_1\ge 12$. There exist $R_1,R_2\in(0,1)$ such that $r_2=\frac{R_2}{(1-R_2)^2}$ and $r_1=\frac{R_1}{(1+R_1)^2}$. If $(1-R_2)^2\le 1/3$, then  $R_2=r_2(1-R_2)^2<1/3$, which then implies $(1-R_2)^2>(1-1/3)^2>1/3$, which is a contradiction. So $(1-R_2)^2>1/3$, and from $R_1<1$, we have
	$\frac{R_2}{R_1}=  \frac{(1-R_2)^2}{(1+R_1)^2}\cdot\frac{r_2}{r_1} >\frac{1/3}{4} \frac{r_2}{r_1} =\frac 1{12}\cdot \frac{r_2}{r_1}$. Since $r_2/r_1>12$, we get $R_2>R_1$. Formula (\ref{i2}) now follows from (\ref{eqn-Koebe}).
	\end{proof}

Let $\lambda_t$ and $q_t$ be the solutions of (\ref{SDE-lambda-q}). Let $X_t=\lambda_t-q_t$. Then $X_t$ is a random process staying in $(0,2\pi)$ that satisfies the SDE
\BGE dX_t=\sqrt{\kappa}dB_t+2\cot(X_t/2)dt.\label{diffusion}\EDE

\begin{Proposition}
	The process $(X_t)$ in (\ref{diffusion}) is Markovian with a transition density $p_t(x,y)$. Moreover, there exist constants $t_0\ge 1$ and  $C_1,C_2>0$ depending only on $\kappa$ such that for $t\ge t_0$, $$ C_1\sin(y/2)^{8/\kappa}\le p_t(x,y)\le C_2 \sin(y/2)^{8/\kappa},\quad \forall x,y\in(0,2\pi).$$
	\label{diffusion-estimate}
\end{Proposition}
\begin{proof}
	Let $Y_t=\cos(\frac 12 X_{\frac 4\kappa \cdot t})$. Then for some standard Brownian motion $\til B_t$, $(Y_t)$ solves the SDE
	\BGE dY_t=-\sqrt{1-Y_t^2}d\til B_t-\Big(\frac 4\kappa+\frac 12\Big) Y_tdt.\label{SDE-Y}\EDE
	This agrees with \cite[Formula (8.2)]{tip} with $\delta=1+\frac 8\kappa$. From \cite[Proposition 8.1]{tip}, $(Y_t)$ has a transition density $p^{(Y)}_t(x,y)$ given by
	$$ p^{(Y)}_t(x,y)=\sum_{n=0}^\infty \frac {(1-y^2)^{\frac 4\kappa -\frac 12} C_n^{(\frac 4\kappa)}(x) C_n^{(\frac 4\kappa)}(y)}{\langle C_n^{(\frac 4\kappa)},C_n^{(\frac 4\kappa)}\rangle_{\frac 4\kappa -\frac 12}}\, e^{-\frac n2(n+\frac 8\kappa)t},$$
	where $C_n^{(\frac 4\kappa)}$, $n=0,1,2,\dots$, are the Gegenbauer polynomials \cite{orthogonal} with  index $\frac 4\kappa$, and $\langle\cdot \rangle_{\frac 4\kappa -\frac 12}$ is the inner product defined by
	$\langle f,g \rangle_{\frac 4\kappa -\frac 12}=\int_{-1}^1 (1-x^2)^{\frac 4\kappa -\frac 12} f(x)g(x) dx$.
	
	Since $Y_t=\cos(X_{\frac 4\kappa \cdot t}/2)$, $(X_t)$ also has a transition density, which is
	$$p_t(x,y)= \sum_{n=0}^\infty \frac {\sin(y/2)^{\frac 8\kappa} C_n^{(\frac 4\kappa)}(\cos(x/2)) C_n^{(\frac 4\kappa)}(\cos(y/2))}{2\langle C_n^{(\frac 4\kappa)},C_n^{(\frac 4\kappa)}\rangle_{\frac 4\kappa -\frac 12}}\, e^{-n(1+\frac \kappa 8n)t}.$$
	
	From \cite{orthogonal} we know that $C_0^{(\frac 4\kappa)}\equiv 1$, and for all $n\ge 0$,
	$$\langle C_n^{(\frac 4\kappa)},C_n^{(\frac 4\kappa)}\rangle_{\frac 4\kappa -\frac 12}=\frac{\pi \Gamma(n+\frac 8\kappa)}{2^{\frac 8\kappa -1}(n+\frac 4\kappa)n!\Gamma(\frac 4\kappa)^2},\quad
	\| C_n^{(\frac 4\kappa)}\|_\infty=\frac{\Gamma(n+\frac 8\kappa)}{n!\Gamma(\frac 8\kappa)}.$$
	Thus, we have
	$$\Big|\frac{\kappa\pi \Gamma(\frac 8\kappa)} { {2^{\frac 8\kappa}} \Gamma(\frac 4\kappa)^2}\cdot\frac {p_t(x,y)}{\sin(y/2)^{\frac 8\kappa}}-1 \Big|\le \sum_{n=1}^\infty  \frac{(\frac \kappa 4 n+1)\Gamma(n+\frac 8\kappa)}{ n!\Gamma(\frac 8\kappa)} \, e^{-n(1+\frac \kappa 8n)t}. $$
	Using Stirling's formula, we can conclude that there exists $t_0\ge 1$ such that for $t\ge t_0$, the RHS of the above formula is less than $1/2$, which implies the conclusion.
\end{proof}

\begin{proof} [Proof of Lemma \ref{plrR2*}.]
	Using Proposition \ref{conformal-content}, Koebe distortion theorem and Proposition \ref{Lem-Koebe} (i) (applied to a conformal map between $\D^*$ and $D\in{\cal D}$ that fixes $\infty$, precomposed and postcomposed  by the reciprocal function $z\mapsto 1/z$), we find that Lemma \ref{plrR2*} follows from the following statement:
\BGE\exists R>r>6,\,\forall p>0,\, \exists l>0,\, \forall z\in\TT\sem\{1\},\quad \nu^\#_{\D^*;z\to \infty;1}(\{\beta:\Cont(\beta([\tau^R_r,\tau_R]))<l \})<p.\label{temp1}\EDE

To see this, first by scaling we may assume that $\ccap(D^c)=0$. Let $g$ be a conformal map from $D$ onto $\D^*$ with $g(\infty)=\infty$. Let $J(z)=1/z$ and $f=J\circ g^{-1}\circ J$. Then $f\in{\cal U}_1$. Suppose (\ref{temp1}) holds with some $R>r>6$. Then by Proposition \ref{Lem-Koebe} (i), there are $0<r_1<r_2<\frac{3/2}r$ such that
$\{r_1<|z|<r_2\}\supset f(\{1/R<|z|<1/r\})$. Let $\til R=1/r_1$ and $\til r=1/r_2$. Then $\til R>\til r>\frac 23 r>4$ and $\{\til r<|z|<\til R\}\supset g^{-1}(\{r<|z|<R\})$. We may then check that Lemma \ref{plrR2*} holds for such $\til R$ and $\til r$. Let $p>0$. Let $l>0$ be given by (\ref{temp1}) for such $p$ (and the $R$ and $r$). Let $D\in\cal D$ have $\ccap(D^c)=0$, and let $a,b$ be distinct prime ends of $D$. We may assume that the above $g$ satisfies $g(b)=1$. Then $g(a)\in\TT\sem \{1\}$. Suppose $\beta$ has the law $\nu^\#_{\D^*;g(a)\to \infty;1}$. Then $\til \beta:=g^{-1}(\beta)$ has the law $\nu^\#_{D;a\to\infty; b}$. From (\ref{temp1}) we know that $\PP[\Cont(\beta([\tau^R_r(\beta),\tau_R(\beta)])<l]<p$. Since $|\log|(g^{-1})'||$ is uniformly bounded on $\{r<|z|<R\}$, by Proposition \ref{conformal-content}, there is $\til l\in(0,l)$ such that $\Cont(\til \beta([\tau^R_r(\beta),\tau_R(\beta)])<\til l$ implies that $\Cont( \beta([\tau^R_r(\beta),\tau_R(\beta)])< l$. From $\{\til r<|z|<\til R\}\supset g^{-1}(\{r<|z|<R\})$ we get $\tau_{\til R}(\til \beta)\ge \tau_R(\beta)$ and $\tau^{\til R}_{\til r}(\til\beta)\le \tau^R_r(\beta)$, which implies that $\PP[\Cont(\til\beta([\tau^{\til R}_{\til r}(\til\beta),\tau_{\til R}(\til \beta)])<\til l]\le \PP[\til \beta([\tau^R_r(\beta),\tau_R(\beta)])<\til l]<p$, as desired.

	Let $z\in\TT\sem\{1\}$. Let $\beta^z$ be a radial SLE$_\kappa(2)$ curve in $\D^*$ from $z$ to $\infty$  with a force point at $1$ (parametrized by capacity). Let $\lambda_t$ and $q_t$ be the driving process and force point process for $\beta^z$, respectively; let $D_t$ be the complement of $\D^*\sem\beta^z([0,t])$ that contains $\infty$; and let $g_t$ be the radial Loewner maps for $\beta^z$. Then $\lambda_t$ and $q_t$ solve (\ref{SDE-lambda-q}) with $x\in(0,2\pi)$ such that $e^{ix}=z$ and $y=0$. Let $X_t=\lambda_t-q_t$, which is a diffusion process that solves (\ref{diffusion}) with initial value $X_0=x$. Let $t_0\ge 1$  be as in Proposition \ref{diffusion-estimate}. Then the law of $X_{t_0}$ has a density w.r.t.\  $\mA|_{(0,2\pi)}$, which is comparable to $\sin(y/2)^{\frac 8\kappa}$.
	
	Let $(\F_t)$ be the filtration generated by $\beta^z$.
	Since $e^{t_0}\ge e>2$, the  statement   (\ref{temp1}) follows from the following:
	\BGE\exists R>r>4e^{t_0},\,\forall p>0,\,\exists l>0,\, \forall z\in\TT\sem\{1\},\quad \PP[ \Cont(\beta^z([\tau^R_r,\tau_R]))<l|\F_{t_0} ]<p.\label{temp1'}\EDE
	
	Let $\ha g_{t_0}(z)=g_{t_0}(z)/e^{iq_{t_0}}$. Then $\ha g_{t_0}$ maps $D_{t_0}$ conformally onto $\D^*$, fixes $\infty$ and $1$, and maps $\beta^z(t_0)$ to $e^{iX_{t_0}}$. From the domain Markov property for $\beta^z$,  the conditional law  given $\F_{t_0}$  of the $\ha g_{t_0}$-image of the  part of $\beta^z$ after $t_0$, denoted by $\ha\beta^z_{t_0}$, is $\nu^\#_{\D^*;X_{t_0}\to \infty;1}$. Using Proposition \ref{conformal-content}, Koebe distortion theorem and Proposition \ref{Lem-Koebe} (i) applied to  $a=e^{-t_0}$ and $f=1/\ha g_{t_0}^{-1}(1/z)$, and using the fact that $\tau^R_r>t_0$ since $\diam(\beta^z([0,t_0]))\le 4e^{t_0}<r<R$, we find that the  statement  (\ref{temp1'}) follows from the following:
	\BGE \exists R>r>6,\,\forall p>0,\,\exists l>0,\, \forall z\in\TT\sem\{1\},\quad \PP[ \Cont(\ha\beta^z_{t_0}([\tau^R_r,\tau_R]))<l|\F_{t_0} ]<p.\label{temp1''}\EDE
Here we used an argument that is similar to that used right after the statement (\ref{temp1}). Note that if (\ref{temp1''}) holds for some $R>r>6$, then we may find $\til R>\til r>4 e^{t_0}$ such that $\{\til r<|z|<\til R\}\supset \ha g_{t_0}^{-1}(\{r<|z|<R\})$. This has to do with that $\ha g_{t_0}(z)/z\to e^{-t_0}$ as $z\to \infty$.
	
Let $\til\beta$ be a radial SLE$_\kappa(2)$ curve  in $\D^*$ from a random point $e^{i\theta}$ to $\infty$ with a  force point at $1$, where $\theta$ is distributed on $(0,2\pi)$ with a density function proportional to $\sin(x/2)^{\frac 8\kappa}$ w.r.t.\ the Lebesgue measure. Since the Minkowski content of any subcurve of $\til\beta$ crossing $\{7\le |z|\le 8\}$ is a.s.\ strictly positive, we find that
\BGE \forall p>0,\quad \exists l>0,\quad\PP[\Cont(\til \beta([\tau^8_7,\tau_8]))<l]<p.\label{temp1'''}\EDE
Since the conditional law of $X_{t_0}$ given $\F_{t_0}$ is absolutely continuous w.r.t.\ the law of ${\theta}$, and the Radon-Nikodym derivative is uniformly bounded away from $\infty$ and $0$, the same is true for the relation between the conditional law of $\ha\beta^z_{t_0}$ given $\F_{t_0}$ and the law of $\til\beta$. So from (\ref{temp1'''})  we find that (\ref{temp1''}) holds with $R=8$ and $r=7$. The proof is now complete.
\end{proof}

\begin{proof} [Proof of Lemma \ref{plrR3*}.]
	Using Proposition \ref{conformal-content}, Koebe distortion theorem and Proposition \ref{Lem-Koebe} (ii) (applied to a conformal map between $\D^*$ and $D\in{\cal D}$ that fixes $\infty$, precomposed and postcomposed by the reciprocal function), we find that the statement of the lemma follows from the following statement:
	\BGE\exists R>6r>6,\,\forall l>0,\,\exists p>0,\,\forall z\in\TT\sem\{1\},\quad \nu^\#_{\D^*;z\to \infty;1}(\{\beta:\Cont(\beta([\tau_r^R,\tau_R]))<l \})>p.\label{temp2}\EDE
Here we used an argument that is similar to the one used after (\ref{temp1}).

	Let $z\in\TT\sem\{1\}$. Let $t_0\ge 1$  be as in Proposition \ref{diffusion-estimate}. Let $\beta^z$, $(X_t)$, $(\F_t)$, $\ha g_{t_0}$ and $\ha\beta^z_{t_0}$ be as in the proof of Lemma \ref{plrR2*}. 	Using Proposition \ref{conformal-content}, Koebe distortion theorem and Proposition \ref{Lem-Koebe} (ii) applied to  $a=e^{-t_0}$ and $f=1/\ha g_{t_0}^{-1}(1/z)$, we find that the  statement with (\ref{temp2}) follows from the following: there exist $R>r>1$ with $R/r>6^2$ such that for any $l>0$ there exists $p>0$ such that for any $z\in\TT\sem\{1\}$,
	\BGE \PP[ \Cont(\ha\beta^z_{t_0}([\tau^R_r,\tau_R]))<l|\F_{t_0} ]>p.\label{temp2'}\EDE
	
Since the conditional law of $\ha\beta^z_{t_0}$ given $\F_{t_0}$ is $\nu^\#_{\D^*;e^{iX_{t_0}}\to \infty;1}$, and the law of $X_{t_0}$ has a density w.r.t.\  $\mA|_{(0,2\pi)}$, which is bounded below on $[\frac 23\pi,\frac 43\pi]$ by a uniform positive constant, the statement with (\ref{temp2'}) follows from the following: there exist $\eps\in(0,1)$ and $R>r>1$ with $R/r\ge 6^2$  such that for any $l>0$ there exists $p>0$ such that for any  $z\in\TT$ with $|z+1|<\eps$,
$$ \nu^\#_{\D^*;z\to \infty;1}(\{\beta:\Cont(\beta([\tau_r^R,\tau_R]))<l \})>p.$$
Using reciprocal maps, we convert this statement into the following:
there exist $\eps<R<r\in(0,1)$ with $R/r< 1/6^2$  such that for any $l>0$ there exists $p>0$ such that for any  $z\in\TT$ with $|z+1|<\eps$,
\BGE \nu^\#_{\D;z\to 0;1}(\{\beta:\Cont(\beta([\tau_r^R,\tau_R]))<l \})>p.\label{temp2''}\EDE

If $z\in\TT$ satisfies $|z+1|<\eps$ for some $\eps<1$, then there exists a M\"obius transformation that maps $(\D;z,0,1)$   onto $(\D;-1,w,1)$ for some $w\in\C$ with $|w|<\eps$. By applying Koebe distortion theorem to the map and using  Proposition \ref{conformal-content}, we see that the statement with (\ref{temp2''}) follows from the following: there exist $0<\eps<R<r<1$ with $\eps/R<1/2$ and $R/r< 1/6^3$ such that for any $l>0$ there exists $p>0$ such that for any  $w\in\C$ with $|w|< \eps$,
\BGE\nu^\#_{\D;-1\to w;1}(\{\beta:\Cont(\beta([\tau_r^R,\tau_R]))<l \})>p.\label{temp2'''}\EDE

We claim that for any $0<\eps<R<1$ with $\eps/R<1/2$ and for any $w\in\C$ with $|w|< \eps$,  $\nu^\#_{\D;-1\to w;1}|_{\F_{\tau_R}}$ is absolutely continuous w.r.t.\ $\nu^\#_{\D;-1\to 0;1}|_{\F_{\tau_R}}$, and the Radon-Nikodym derivative is uniformly bounded above and below by finite positive constants depending only on $\kappa,\eps, R$. To see that this is true, we may compare $\nu^\#_{\D;-1\to w;1}|_{\F_{\tau_R}}$ with the chordal SLE$_\kappa$ measure $\mu^\#_{\D;-1\to 1}$ restricted to the $\sigma$-algebra ${\F_{\tau_R}}$. We use the fact that $\nu^\#_{\D;-1\to w;1}|_{\F_{\tau_R}}$ is absolutely continuous w.r.t.\ $\mu^\#_{D;-1\to 1}|_{\F_{\tau_R}}$, and the Radon-Nikodym derivative equals ${\bf 1}_{\{\tau_R<\infty\}}G_{\D_{\tau_R};\beta(\tau_R),1}(w)/G_{\D;-1,1}(w)$ (cf.\ \cite[Proposition 2.12]{LW}). To prove the claim, we need to show that there is a constant $C\in(1,\infty)$ depending on $\kappa,\eps,R$ such that for any simply connected domain $D$ that contains the disc $\{|z|<R\}$, for any distinct prime ends $a,b$ of $D$, and for any $w\in\C$ with $|w|\le \eps$, we have $1/C\le G_{D;a,b}(w)/G_{D;a,b}(0)\le C$. For this purpose, we use the explicit expression of the chordal SLE$_\kappa$ Green's function  (cf.\ \cite[Formula (8)]{LR}): \BGE G_{D;a,b}(w)= \ha c\Upsilon_{D}(w)^{1-\frac \kappa 8} \sin(\pi*h_D(w;\stackrel{\frown}{ab}))^{\frac 8\kappa -1},\label{Green-expression}\EDE where $\ha c$ is a positive constant depending only on $\kappa$, $ \Upsilon_{D}(w)$ is the conformal radius of $D$ viewed from $w$, and $h_D(w;\stackrel{\frown}{ab})$ is the harmonic measure of the clockwise arc on $\pa D$ from $a$ to $b$ viewed from $w$. Using Koebe distortion theorem and Harnack inequality, we can prove that, if $|w|\le \eps$, then $\Upsilon_D(w)\asymp \Upsilon_D(0)$ and $\sin(\pi*h_D(w;\stackrel{\frown}{ab}))\asymp \sin(\pi*h_D(0;\stackrel{\frown}{ab}))$, and the implicit constants depend only on $\kappa,\eps,R$. Then the claim is proved.

Thus, the statement with (\ref{temp2'''}) follows from the following: there exist $0<R<r<1$ with $R/r\le 1/6^3$  such that for any $l>0$,
$$ \nu^\#_{\D;-1\to 0;1}(\{\beta:\Cont(\beta([\tau^R_r,\tau_R]))<l \})>0.$$ 
It will be more convenient to work on square domains. Let $Q$  be the open square $\{z\in \C:|\Ree z|, |\Imm z|< 1\}$. For $R,r>0$, we define $\ha\tau_R$ to be the hitting time (i.e., the first visiting time) of $\pa (RQ)$, and define $\ha\tau^R_r$ to be the last visiting time of   $\pa (rQ)$ before $\ha\tau_R$. Using Proposition \ref{conformal-content}, Koebe distortion theorem and Proposition \ref{Lem-Koebe} (iii), we find that the previous statement   follows from the following: there exist $R<r\in(0,1)$ with $R/r\le 1/(12*\sqrt 2 *6^3)$  such that
$$ \forall l>0,\quad\nu^\#_{Q;-1\to 0;1}(\{\beta:\Cont(\beta([\ha\tau^R_r,\ha\tau_R]))<l \})>0.$$ 
Since $\nu^\#_{Q;-1\to 0;1}|_{\F_{\ha\tau_R}}$ is absolutely continuous w.r.t.\ the chordal SLE$_\kappa$ measure $\mu^\#_{Q;-1\to 1}|_{\F_{\ha\tau_R}}$, and the Radon-Nikodym derivative is strictly positive on the event $\{\ha\tau_R<\infty\}$, the above statement follows from
\BGE \forall R,l\in(0,1),\quad \mu^\#_{Q;-1\to  1}(\{\beta:\ha\tau_R(\beta)<\infty,\Cont(\beta([0,\ha\tau_R]) )<l \})>0.  \label{for-Remark}\EDE
Define $U_r=\{z\in\C:|\Ree z|< 1,|\Imm z|<r\}$ for $r\in(0,1]$. 
From the generalized restriction property for SLE$_\kappa$, $\kappa\in(0,4]$,  we know that, for any $\delta\in(0,1)$, $ \mu^\#_{U_\delta;-1\to  1} $ is  absolutely continuous w.r.t.\ $ \mu^\#_{Q;-1\to  1}|_{\{\cdot\cap(Q\sem U_\delta) =\emptyset\}}$, and the Radon-Nikodym derivative is expressed in terms of Brownian loop measures and the central charge for SLE$_\kappa$ (\cite{LSW-8/3,loop-soup}). So (\ref{for-Remark}) follows from
\BGE \forall R,l\in(0,1),\exists\delta\in(0,R),\quad \mu^\#_{U_\delta;-1\to  1}(\{\beta: \Cont(\beta([0,\ha\tau_R]) )<l \})>0. \label{temp4}\EDE
Using  Formula (\ref{Green-expression}), Koebe $1/4$ theorem, and the fact that the expectation of the Minkowski content measure of a chordal SLE$_\kappa$ curve equals the integral of its Green's function (cf.\ \cite[Theorem 1.1]{LR}),  we get  $$\int \Cont(\beta) \mu^\#_{U_\delta;-1\to  1}(d\beta)=\int_{U_\delta} G_{U_\delta;-1,1}(z)\mA^2(dz)\lesssim \int_{U_\delta} \dist(z,U_\delta^c)^{\frac{\kappa}{8}-1}\mA^2(dz).$$
Let $Q_1$ be the first quadrant $\{\Ree z,\Imm z>0\}$. For $z\in U_\delta\cap Q_1$, we have $\dist(z,U_\delta^c)=\min\{1-\Ree z,\delta-\Imm z\}$. Since $\frac{\kappa}{8}-1<0$, we have
$$\int_{U_\delta\cap Q_1} \dist(z,U_\delta^c)^{\frac{\kappa}{8}-1}\mA^2(dz)
\le \int_0^1 \int_0^\delta (1-x)^{\frac{\kappa}{8}-1}+(\delta-y)^{\frac{\kappa}{8}-1} dxdy=\frac 8\kappa(\delta+\delta^{\frac \kappa 8}).$$
By symmetry, this inequality also holds with $Q_1$ replaced by other quadrants. So
we have $\int \Cont(\beta) \mu^\#_{U_\delta;-1\to  1}(d\beta)\lesssim  \delta^{\frac \kappa 8}$. We can make $\int \Cont(\beta) \mu^\#_{U_\delta;-1\to  1}(d\beta)<l$ by choosing $\delta\in(0,R)$ small enough. So  we get (\ref{temp4}) and finish the proof.
\end{proof}

\begin{Remark}
One step of the above argument does not work for $\kappa\in(4,8)$, i.e., the generalized restriction property does not hold for $\kappa>4$. One possible way to get around this issue is to work on a chordal SLE$_\kappa$ curve in $U_\delta$ from $-1$ to $1$ conditioned to avoid the two horizontal boundary segments of $U_\delta$. This law of such SLE is absolutely continuous w.r.t.\ $ \mu^\#_{Q;-1\to  1}|_{\{\cdot\cap(Q\sem U_\delta)=\emptyset\}}$. If we can prove that the integral of the Green's function for the conditional SLE$_\kappa$ curve in $U_\delta$ tends to $0$ as $\delta\to 0$, then we can follow the above proof to extend Lemma \ref{plrR3*} and Theorem \ref{not-Holder} to $\kappa\in(4,8)$. This requires some estimate of the Green's functions for this SLE curve, such as $G(z)\lesssim\dist(z,\pa U_\delta)$, which we do not have now.
\end{Remark}

\begin{Remark}
For $\kappa\ge 8$,  we   have a self-similar SLE$_\kappa$ process   of index $1/2$ with stationary increments. This is the SLE$_\kappa$ loop rooted at $\infty$ (now we have a probability measure, cf.\ \cite[Remark 4.8]{loop}). It is space-filling (visits every point in $\C$), and parametrized by the Lebesgue measure.  We can prove that Lemma \ref{integrable*} still holds for $c\in(-2,0)$. To see this, note that (\ref{int-infty*}) holds because
$$\int_0^\infty f_c(\gamma(t))dt\le \int_{-\infty}^{\infty}f_c(\gamma(t))dt=\int_{\C} f_c(z)\mA^2(dz)=\int_{|z|\le 1} |z|^c\mA^2(dz)<\infty.$$
We may decompose an SLE$_\kappa$ loop $\gamma$ rooted at $\infty$ into two arms.
The first arm $\gamma_-:=\gamma|_{\R_-}$ is a whole-plane SLE$_\kappa(\kappa-6)$ curve from $\infty$ to $0$; and given $\gamma_-$, the second arm $\gamma_+:=\gamma|_{\R_+}$ is a chordal SLE$_\kappa$ curve from $0$ to $\infty$ in $\ha\C\sem \gamma_-$. For $\kappa=8$, from the reversibility of whole-plane SLE$_8(2)$ (\cite{MS4}), we know that the time-reversal $\gamma^R_-$ of $\gamma_-$ is a whole-plane SLE$_8(2)$ curve from $0$ to $\infty$. Then Lemmas \ref{plrR2*} and \ref{plrR3*} also hold for $\kappa=8$. In fact, the proof of Lemma \ref{plrR2*} still works in the case $\kappa=8$ without any change. The proof of Lemma \ref{plrR3*}  works from its start all the way to the statement (\ref{for-Remark}), but (\ref{for-Remark}) should be proved in a different way. We have (\ref{for-Remark}) for $\kappa=8$ because for any $\delta<R<1$, there is a positive probability that a chordal SLE$_8$ curve in $Q$ from $-1$ to $1$ will reach the vertical line $\{\Ree z=-R\}$ before any of the horizontal lines $\{\Imm z=\delta\}$ and $\{\Imm z=-\delta\}$, and when this happens, $\Cont(\beta([0,\ha\tau_R]))=\mA(\beta([0,\ha\tau_R]))\le 2\delta(1-R)<l$  if $\delta$ is  small enough.
Thus, Theorems \ref{Theorem-Holder*}, \ref{Theorem-Mckean*} and \ref{not-Holder} all hold in the case $\kappa=8$. The above argument does not work for $\kappa> 8$ due to the lack of time-reversibility of whole-plane SLE$_\kappa(\kappa-6)$.
\label{Remark-last}
\end{Remark}

\begin{Remark}
Mohammad Rezaei proved in \cite{R} that for $\kappa\in(0,8)$, the $d$-dimensional Hausdorff measure of an SLE$_\kappa$ curve is $0$. Using the sssi SLE$_\kappa$ curve, we may gain another perspective on that statement. We now need a new crossing estimate: there exist $R>r>4$ such that for any $L>0$ there exists $p>0$ such that for any $D\in {\cal D}$ with distinct prime ends $a,b$,
\BGE \nu^\#_{D;a\to \infty;b}(\{\beta:\Cont(\beta([\tau^{Re^{\ccap(D^c)}}_{re^{\ccap(D^c)}},\tau_{Re^{\ccap(D^c)}}]))>L*e^{d*\ccap(D^c)}\})>p.\label{new}\EDE
The estimate is similar to \cite[Lemma 3.2]{R}, and its proof should also follow the proof of that lemma, which is technically involved. 
With this estimate, one may prove that  $\gamma([0,1])$ has Hausdorff measure zero using the following approach. For $t\in  [0,1]$ and $n\in\N$,  let $Q_t^n$ denote the dyadic square containing $\gamma(t)$ of side length $2^{-n}$.  Fix $\eps>0$ and $M>m\in\N$. For each $t$, define $Q_t$ such that if there is $n\in[m,M]$ such that $\diam(Q_t^n)^d< \eps \Cont(\gamma\cap Q^n_t)$, then let $Q_t$ be the largest of such $Q_t^n$; otherwise let $Q_t=Q_t^M$.  We then get a finite covering $\{Q_t:t\in[0,1]\}$ of $\gamma([0,1]) $. The sum of the $d$-th powers of the diameters of these squares is naturally decomposed into two parts. The first is over the squares with side length $>2^{-M}$, whose sum is less than $\eps \Cont(\gamma([0,1]))$. The second is over the squares with side length $2^{-M}$. Since $\gamma([0,1])$ has Minkowski content $1$, the typical number of the dyadic squares with side length $2^{-M}$ that intersect $\gamma([0,1])$ is comparable to $2^{dM}$. Given one of such square, the probability that it belongs to the covering $\{Q_t:t\in[0,1]\}$ is small when $M$ is big by (\ref{new}). So we can estimate the expectation of the second sum, and show that it tends to zero as $M\to \infty$.
\end{Remark}

\end{document}